\documentclass[11pt,twoside, reqno]{amsart}
\usepackage{amsfonts}
\usepackage{latexsym}
\usepackage{amssymb}
\usepackage{amsmath}
\usepackage{dsfont}
\usepackage{multicol} 
\usepackage[pdftex,svgnames]{xcolor}
\usepackage{url}
\usepackage[utf8]{inputenc}
\usepackage{mathtools}
\makeatletter
\@namedef{subjclassname@2020}{\textup{2020} Mathematics Subject Classification}
\makeatother

\renewcommand{\le}{\leqslant}
\renewcommand{\leq}{\leqslant}
\renewcommand{\ge}{\geqslant}
\renewcommand{\geq}{\geqslant}

\newcommand{\al}{\alpha}

\newcommand{\N}{{\mathbb N}}
\newcommand{\R}{{\mathbb R}}

\newcommand{\Fin}{\mathfrak{Fin}}
\newcommand{\I}{\mathfrak{I}}

\theoremstyle{plain}
\newtheorem{theorem}{Theorem}
\newtheorem{lemma}[theorem]{Lemma}

\theoremstyle{remark}

\theoremstyle{definition}

\numberwithin{equation}{section}

\begin{document}
	\title[On ideal generated by modulus function]{On relation between statistical ideal and ideal generated by a modulus function}
	\author{Dmytro Seliutin}
	\address{School of Mathematics and Informatics V.N. Karazin Kharkiv  National University,  61022 Kharkiv, Ukraine}
	\email{selyutind1996@gmail.com}
	\thanks{ The research  was supported by the National Research Foundation of Ukraine funded by Ukrainian state budget in frames of project 2020.02/0096 ``Operators in infinite-dimensional spaces:  the interplay between geometry, algebra and topology''}

\begin{abstract}
We characterize those modulus functions $f$ for which the ideal generated by $f$ is equal to the statistical ideal.
\end{abstract}
	
	\maketitle
	
	\section{Introduction}
	
	Let $\Omega$ be a non-empty set. Let us remind that a non-empty family $\I \subset 2^{\Omega}$ is called \textit{an ideal} on $\Omega$ if $\I$ satisfies:
	\begin{enumerate}
	    \item $\Omega \notin \I$;
	    \item if $A, B \in \I$ then $A \cup B \in \I$;
	    \item if $A \in \I$ and $D \subset A$ then $D \in \I$.
	\end{enumerate}
	
In our article we consider those ideals $\I$ which contain the family of finite sets $\Fin$.
	
For a subset $A \subset \N$ denote $\al_A(n) := |A \cap [1,n]|$, where $|M|$ stands for a number of elements in the set $M \subset \N$.  Let $A \subset \N$. \textit{The natural density} of $A$ is $\displaystyle d(A) := \lim\limits_{n \rightarrow \infty} \frac{\al_A(n)}{n}$.
	
	The ideal of sets $A \subset \N$ having $d(A) = 0$ is called \textit{the statistical ideal}. We denote this ideal $\I_s$.
	
	The statistical ideal is related to the statistical convergence and is a very popular branch of research.
	
	In \cite{aizpuru2014} authors introduced a generalization of the natural density of subset in $\N$. They called it \textit{$f$-density}, where $f$ is a modulus function.
	
	Recall that a function $f: \R^+ \rightarrow \R^+$ is called \textit{an unbounded modulus function} (modulus function for short) if:
		\begin{enumerate}
			\item $f(x) = 0$ if and only if $x = 0$;
			\item $f(x + y) \leq f(x) + f(y)$ for all $x, y \in \R^+$;
			\item $f(x) \le f(y)$  if $x \le y$;
			\item $f$ is continuous from the right at 0;
			\item $\lim\limits_{n \rightarrow \infty} f(n) = \infty$.
		\end{enumerate}

		Let $f$ be a modulus function. The quantity  $\displaystyle d_f(A) := \lim\limits_{n \rightarrow \infty} \frac{f(\al_A(n))}{f(n)}$ is called \textit{the $f$-density} of $A \subset \N$. The ideal $\I_f := \{A \subset \N:\ d_f(A) = 0\}$ is called \textit{the $f$-ideal}. $\I_f$ appears implicitly in \cite{aizpuru2014} where the convergence of sequences with respect to $\I_f$ was studied, and appears explicitly in \cite{kad-sel}.
		
In \cite[p. 527]{aizpuru2014} it is noted that for an arbitrary modulus function $f$ and $A \subset \N$ if  $d_f(A) = 0$ then $d(A) = 0$. In other words, $\I_f \subset \I_s$. The ideals $\I_f$ and $\I_s$ and the corresponding ideal convergences have some similarities and some differences. The aim of the paper is to present a complete description of those modulus functions $f$ for which $\I_f = \I_s$.  We do this in Theorem \ref{main_theorem}. After that in Theorem \ref{easier} we give a handy sufficient condition for the equality $\I_s = \I_f$, and finally we present some illustrative examples.

	\section{Main results}
	
Let $f$ be a modulus function, $t \in [1, +\infty)$, $k \in \N$. Denote 
	$$
	h_f(t) := \displaystyle\limsup\limits_{n \rightarrow \infty} \frac{f(n)}{f(tn)}, \quad g_f(k) := h_f(2^k) = \displaystyle\limsup\limits_{n \rightarrow \infty} \frac{f(n)}{f(2^k n)}.
	$$
	
The following Theorem is the promised main result of our paper. 
	
\begin{theorem} \label{main_theorem}
Let $f$ be a modulus. The following statements are equivalent:
	\begin{enumerate}
	    \item $\I_s = \I_f$;
	    \item $\lim\limits_{k \rightarrow \infty} g_f(k) = 0$.
	    \item $\lim\limits_{t \rightarrow \infty} h_f(t) = 0$.
\end{enumerate}
\end{theorem}
	
	\begin{proof} Remark that the equivalence of conditions (2) and (3) follows evidently from the monotonicity of $h_f(t)$ in the variable $t$. We include both of the conditions because of some conveniences in the proof and for the future applications. So, it is sufficient to demonstrate implications  $(3) \Rightarrow (1)$ and $(1) \Rightarrow (2)$.

$\mathbf{(3) \Rightarrow (1)}$. As we remarked in the Introduction, the inclusion $\I_s \supset \I_f$ is known, so we need to show that $\I_s \subset \I_f$. Denote $\displaystyle \delta^f_t := h_f(t) + \frac{1}{t}$. The quantity $\delta^f_t$ is  decreasing in $t$ and $\lim\limits_{t \rightarrow \infty} \delta^f_t = 0$. We know that $\displaystyle\limsup\limits_{n \rightarrow \infty} \frac{f(n)}{f(t n)} < \delta^f_t$ for all $t \in [1, +\infty)$, in particular we have that for every $k \in \N$ there exists $N_1(k) \in \N$ such that for all $n \geq N_1(k)$ the following holds true: $\displaystyle f\left(\frac{n}{k}\right) < \delta^f_k f(n)$.
	
	Let $A \in \I_s$. By the definition of $\I_s$ his means that $\displaystyle\lim\limits_{n \rightarrow \infty} \frac{\al_A(n)}{n} = 0$. In other words,  for each $k \in \N$ there exists $N_2(k) \in \N$ such that $\displaystyle\al_A(n) < \frac{n}{k}$ for each $n > N_2(k)$. Denote $\displaystyle{N_k} := \max \left\{N_1(k), N_2\left(k\right)\right\}$. Then for each $n > {N_k}$ 
	$$\displaystyle f(\al_A(n)) < f \left(\frac{n}{k}\right) < \delta^f_k  f(n).
	$$
	From the previous inequality we have:
	$$
	\limsup\limits_{n \rightarrow \infty} \frac{f(\al_A(n))}{f(n)} \leq \delta^f_k \underset{k \rightarrow \infty}{\longrightarrow} 0,
	$$
which completes the proof of the implication  $(3) \Rightarrow (1)$.

\smallskip	
$\mathbf{(1) \Rightarrow (2)}$.  Assume that $\lim\limits_{k \rightarrow 0} g_f(k) \neq 0$.  By monotonicity, this implies the existence of $\xi > 0$ such that $g_f(k) > \xi$  for every $k \in \N$ . Consequently, $\displaystyle\limsup\limits_{n \rightarrow \infty} \frac{f(n)}{f(2^k n)} > \xi$ for each $k \in \N$. Then for every $k \in \N$ there exists an infinite subset $N_k \subset \N$ such that for each $n \in N_k$ \begin{equation}\label{main}
			f(n) > \xi f(2^k n).
\end{equation}
	Choose $0 = n_0 < n_1 < n_2 < n_3 < ...$ such that  $n_j \in N_j$ for each $j \in \N$. So for each $k \in \N$ 
$$
f(n_k) > \xi f(2^k n_k).
$$
	
Denote $m_k := 2^kn_k$, $k = 1, 2, \ldots$ . Let us consider the following set $A$:
\begin{align*}
A &= \{m_1 - n_1 + 1, m_1 - n_1 + 2, ..., m_1 - 1,\\
&m_2 - n_2 + n_1, m_2 - n_2 + n_1 + 1, m_2 - n_2 + n_1 + 2,... ,m_2 - 1,...\},
\end{align*}
that is from each block of naturals in $(m_{k-1}, m_k]$ we chose the $n_{k} - n_{k-1}$ biggest ones. For the correctness of the definition of $A$ we have to check that $m_k - m_{k-1}  >  n_k -  n_{k-1}$ for every $k \in \N$. Indeed, for all $k \in \N$ we have:
\begin{align*}
m_k - n_k + n_{k-1} - m_{k-1} &= 2^kn_k - n_k + n_{k-1} - 2^{k-1} n_{k-1} =\\
&= n_k (2^k - 1) - n_{k-1} (2^{k-1} - 1) > 0,
\end{align*}
because $n_k > n_{k-1}$ and $2^k - 1 > 2^{k-1} - 1$. 

Denote $\al_n := \al_A(n)$. Let us show that $A \notin \I_f$. By our construction, $\al_{m_k} = n_k$ for all $k \in \N$. Using the inequality (\ref{main}) we obtain that $\displaystyle\frac{f(\al_{m_k})}{f(m_k)} > \xi > 0$, so $\displaystyle\frac{f(\al_n)}{f(n)} \not\rightarrow 0$ as $n \rightarrow \infty$, that's why $A \notin \I_f$.\\
		
Let us finally show that $A \in \I_s$. For every $k \in \N$ we can split the block of naturals $[m_k + 1, m_{k+1}] \cap \N$ as follows: 
\begin{align*}
[m_k + 1, m_{k+1}] \cap \N &= [m_k + 1, m_{k+1} - n_{k+1} + n_k - 1] \cap \N  \\
&\cup [ m_{k+1} - n_{k+1} + n_k, m_{k+1} - 1] \cap \N \cup \{m_{k+1}\}.
\end{align*}
On the initial part of this set for $j \in [m_k + 1, m_{k+1} - n_{k+1} + n_k - 1]$ we have: $\al_j = n_k$ and $\displaystyle\frac{\al_j}{j} = \frac{n_k}{j} \leq \frac{n_k}{m_k + 1} \leq \frac{n_k}{m_k} = \frac{1}{2^k}$.
On the next part, for $j \in [ m_{k+1} - n_{k+1} + n_k, m_{k+1} - 1] = [n_{k+1}(2^{k+1} - 1) + n_k, 2^{k+1}n_{k+1}]$ we have $\al_j = n_k + x_j$, where $1 \leq x_j \leq n_{k+1} - n_k$. Using this, we obtain that $\displaystyle\frac{1}{j} \leq \frac{x_j}{j} \leq \frac{n_{k+1} - n_k}{j}$, and so 
\begin{align*}
\frac{\al_j}{j} = \frac{n_k + x_j}{j} \leq \frac{n_{k+1}}{n_{k+1}(2^{k+1} - 1) + n_k} \leq \frac{n_{k+1}}{n_{k+1}(2^{k+1} - 1)} = \frac{1}{2^{k+1} - 1} < \frac{1}{2^k}.
\end{align*}
At the last point $j  = m_{k+1}$ we have: $\al_j = n_{k+1}$ and $$\frac{\al_j}{j} = \frac{n_{k+1}}{m_{k+1}} = \frac{1}{2^{k+1}} < \frac{1}{2^k}.$$
So for an arbitrary $k \in \N$ and for $j \in [m_k + 1, m_{k + 1}]$ we have $\displaystyle\frac{\al_j}{j} < \frac{1}{2^k}$, in other words $A \in \I_s$.
\end{proof}
	
	Now let us discuss a particular case of Theorem \ref{main_theorem} in which the condition for $\I_f =\I_s$ can be substantially simplified. First, a simple technical lemma.
	
\begin{lemma} \label{lemma}
Let $f$ be a modulus function. Let there exists $\displaystyle \lim\limits_{n \rightarrow \infty} \frac{f(n)}{f(2n)} = a$. Then $g_f(k) =\displaystyle \lim\limits_{n \rightarrow \infty} \frac{f(n)}{f(2^kn)} = a^k$ for all $k \in \N$.
\end{lemma}

\begin{proof}
Let us use method of mathematical induction.

The base of induction: $k = 2$.
$$\displaystyle \lim\limits_{n \rightarrow \infty} \frac{f(n)}{f(4n)} = \displaystyle \lim\limits_{n \rightarrow \infty} \frac{f(n) \cdot f(2n)}{f(2n) \cdot f(4n)} = \displaystyle \lim\limits_{n \rightarrow \infty} \frac{f(n)}{f(2n)} \cdot \displaystyle \lim\limits_{n \rightarrow \infty} \frac{f(2n)}{f(4n)} = a^2.
$$

The inductive transition: $k \rightarrow k+1$.

\smallskip
\noindent$\displaystyle \lim\limits_{n \rightarrow \infty} \frac{f(n)}{f(2^{k+1}n)} = \displaystyle \lim\limits_{n \rightarrow \infty} \frac{f(n)}{f(2^kn)} \cdot \displaystyle \lim\limits_{n \rightarrow \infty} \frac{f(2^k n)}{f(2^{k+1}n)} = a^k \cdot a = a^{k+1}.$
\end{proof}

%
	
\begin{theorem} \label{easier}
Let $f$ be a modulus. Suppose that there exists $\displaystyle \lim\limits_{n \rightarrow \infty} \frac{f(n)}{f(2n)}$. Then the following statements are equivalent:
\begin{enumerate}
\item $\I_s = \I_f$;
\item $\displaystyle\lim\limits_{n \rightarrow \infty} \frac{f(n)}{f(2n)} < 1$.
\end{enumerate}
\end{theorem}
\begin{proof}
Under the assumption of existence of $\displaystyle \lim\limits_{n \rightarrow \infty} \frac{f(n)}{f(2n)}$, Lemma \ref{lemma} gives the equivalence of our condition (2) and the condition (2) of Theorem \ref{main_theorem}.
\end{proof}

\section{Examples}
	
At first, let us show that among the very elementary modulus functions $f$ the both possibilities $\I_f = \I_s$ and $\I_f \neq \I_s$ may happen.

\smallskip	
\noindent \textbf{Example 1}. $f(x) = x^p$, $p \in (0,1]$. For this kind of functions $\I_f = \I_s$. Indeed, $\displaystyle\lim\limits_{n \rightarrow \infty} \frac{f(n)}{f(2n)} = \displaystyle\lim\limits_{n \rightarrow \infty} \frac{n^p}{(2n)^p} = \left(\frac{1}{2}\right)^p < 1$.

\smallskip	
\noindent \textbf{Example 2}. $f(x) = \log (1+x)$. In this case $\I_f \neq \I_s$, because $\displaystyle\lim\limits_{n \rightarrow \infty} \frac{f(n)}{f(2n)} = \displaystyle\lim\limits_{n \rightarrow \infty} \frac{\log(1+n)}{\log(1+2n)} = 1$.

\smallskip	
Our next goal is to show that Theorem \ref{main_theorem} does not reduce to its particular case given in Theorem \ref{easier}, i.e. that there is a modulus functions $f$ for which the limit of $\displaystyle \frac{f(n)}{f(2n)}$ does not exist.

\smallskip	
\noindent \textbf{Example 3}. Put $f(0) = 0$, $f(1) = 1$, $f(2) = 2$. The values of $f$ in the remaining natural numbers we define recurrently: if for some $n \in \N$ the values $f(k)$ are already defined for  $k \in [1, 2^n]$, we define $f(k)$ for  $k = 2^n + \al \in [2^n + 1, 2^{n+1}]$, $\al \in [1, 2^n]$, by means of the formula 
\begin{equation}\label{def_f}
f(2^n + \al) = 
\begin{cases}
f(2^n), \text{if} \ n \in \{1, 3, 5,..\} \\ f(2^n) +f(\al), \text{if}\ n \in \{2, 4, 6,...\}.
\end{cases}
\end{equation}
This defines $f(n)$ for each $n \in \N \cup \{0\}$. In the intermediate points let us define $f$ by means of linear interpolation. Such $f$ is defined for each $x \in \R^+$, is monotonic, continuous, and $f(2^k) := 2^{\lceil \frac{k}{2} \rceil}$ for $k \in \N \cup \{0\}$. 

Let us verify that $f$ is a modulus function. For every $w, z \in \N$ (without lost of generality we consider $w > z$) we have to demonstrate that 
	    \begin{equation}\label{w + z}
	    	f(w + z) \leq f(w) + f(z).
	    \end{equation}
This can be done by induction in $n$, where $n$ is the smallest natural for which $w + z \leq 2^n$.

The base $n = 1$ is straightforward. Suppose now that we already proved \eqref{w + z} for $0 \leq w + z \leq 2^n$ and let us prove it for $2^n < w + z \leq 2^{n+1}$. Denote $w + z = 2^n + \al$, where $\al \in [1, 2^n]$.
\begin{enumerate}
\item Let $n$ be an odd number.  It is clear that there are numbers $\tilde{w}, \tilde{z} \in \N$, $\tilde{w} < w$,  $\tilde{z} < z$ such that $\tilde{w} + \tilde{z} = 2^n$. Then $f(w + z) = f(2^n) = f(\tilde{w} + \tilde{z}) \leq f(\tilde{w}) + f(\tilde{z}) \leq f(w) + f(z)$.
\item Let $n$ be an even  number. Then $f(w + z) = f(2^n + \al) = f(2^n) + f(\al)$.
\begin{enumerate}
\item Let $w \ge 2^n$, then $z \le \al$.  Represent $w$ in the form of $w = 2^n + \beta$. In this case $f(w + z) = f(2^n) + f(\al)$ and $f(w) = f(2^n) + f(\beta)$. Then the inequality (\ref{w + z}) rewrites as $f(\al) \leq f(z) + f(\beta)$ which is true by the inductive assumption.
\item Let $w < 2^n$, which means that  $2^{n-1} < w < 2^n$ and $z > \al$.  Then $f(w) = f(2^{n-1}) = f(2^{n})$, because $n-1$ is odd. Again, in this case the inequality (\ref{w + z}) is equivalent to a simpler one: $f(\al) \leq f(w)$ which is true since $z > \al$.
\end{enumerate}
\end{enumerate}
 So, we proved that the function, defined by (\ref{def_f}) is a modulus function. Consider now the sequence $\displaystyle\frac{f(2^n)}{f(2^{n+1})},\ n = 0, 1, 2, ...$. When $n$ is odd we have $\displaystyle\frac{f(2^n)}{f(2^{n+1})} = 1$ and if $n$ is even we have $\displaystyle\frac{f(2^n)}{f(2^{n+1})} = \frac{1}{2}$. This means that the sequence $\displaystyle\frac{f(n)}{f(2n)}$ has no limit. 
 
By the way, in this example $g_f(k) = \displaystyle\limsup\limits_{n \rightarrow \infty} \frac{f(n)}{f(2^k n)} = \frac{1}{2^{k -1}}$,
so $\I_f = \I_s$.
	
	\vspace{8mm}
	
	\textbf{Acknowledgment.} The author is thankful to his parents for their support and his scientific adviser, professor Vladimir Kadets for his constant help with this project.

\end{document}